\documentclass[12pt]{article}

\usepackage{graphicx,amsthm,amsmath,amstext,amsfonts,amssymb}
\usepackage{mathtools}
\usepackage{fullpage}
 
\usepackage{url}
 
\usepackage[utf8]{inputenc}
 
\usepackage{fourier}

\usepackage{authblk} 

\usepackage{hyperref}
\usepackage{enumitem}

\DeclareMathOperator{\diag}{\operatorname{diag}}

\DeclareMathOperator{\Rad}{\operatorname{Rad}}

\DeclareMathOperator*{\Ast}{\operatorname{\ast}}
\newcommand{\ZZ}{\mathbb{Z}}
\newcommand{\RR}{\mathbb{R}}
\newcommand{\R}{\mathbb{R}}
\newcommand{\QQ}{\mathbb{Q}}
\newcommand{\Q}{\mathbb{Q}}
\newcommand{\CC}{\mathbb{C}}
\newcommand{\NN}{\mathbb{N}}
\newcommand{\HH}{\mathbb{H}}

\newcommand{\Stab}{\operatorname{Stab}}
 \DeclareMathOperator{\End}{End}

\newcommand{\GL}{\operatorname{GL}}

\newcommand{\OO}{\mathcal O}

\newcommand{\Sym}{\operatorname{Sym}}

\newcommand{\Hom}{\operatorname{Hom}}

\DeclareMathOperator{\Tr}{\operatorname{Tr}}
\DeclareMathOperator{\tr}{\operatorname{tr}}
\DeclareMathOperator{\op}{\operatorname{op}}
\DeclareMathOperator{\nred}{\operatorname{nred}}
 
\newcommand{\Cl}{\mathcal{C}\ell}

\newcommand{\eg}{{\it e.g. }}
\newcommand{\ie}{{\it i.e. }}

\newcommand{\cf}{{\it cf. }}

\theoremstyle{plain}

\newtheorem{theorem}{Theorem}[section]

\newtheorem{corollary}[theorem]{Corollary}
\newtheorem{lemma}[theorem]{Lemma}

\theoremstyle{definition}
\newtheorem{defn}[theorem]{Definition} 

\theoremstyle{definition}
\newtheorem*{defn*}{Definition} 

\theoremstyle{remark}
\newtheorem{rem}[theorem]{Remark}

\theoremstyle{remark}

\theoremstyle{remark}

\theoremstyle{remark}
\newtheorem*{ex*}{Example}
\theoremstyle{remark}
\newtheorem*{exs*}{Examples}

\makeatletter
\def\namedlabel#1#2{\begingroup
    #2%
    \def\@currentlabel{#2}%
    \phantomsection\label{#1}\endgroup
}
\makeatother

\title{Computing in arithmetic groups with Vorono\"{\i}'s algorithm}
\author[1]{Renaud Coulangeon}
\author[2]{Gabriele Nebe}
\author[3]{Oliver Braun}
\author[4]{Sebastian Sch\"onnenbeck}
 \affil[1]{Univ. Bordeaux, IMB, UMR 5251,F-33400 Talence, France\\
CNRS, IMB, UMR 5251, F-33400 Talence, France\\
renaud.coulangeon@math.u-bordeaux1.fr}
\affil[2,3,4]{Lehrstuhl D f\"ur Mathematik, RWTH Aachen University,
52056 Aachen, Germany\\
nebe@math.rwth-aachen.de}

\begin{document}

\maketitle

\begin{abstract}
We describe an algorithm, meant to be very general, to compute a presentation of the group of units of an order in a (semi)simple algebra over $\QQ$. Our method is based on a generalisation of Vorono\"i's algorithm for computing perfect forms, combined with Bass-Serre theory. It differs essentially from previously known methods to deal with such questions, \eg for units in quaternion algebras. We illustrate this new algorithm by a series of examples where the computations are carried out completely.
\end{abstract}

\section{Introduction}

Let $\Lambda$ be  an order in  a  (semi-)simple finitely generated 
algebra $A$ over $\QQ$. By definition, this means that
 $\Lambda$ is a subring of $A$ and that, additively, it is a 
free abelian group generated by a basis of $A$ over $\QQ$. 
The determination of the unit group $\Lambda^{\times}$ of such an order
 is in general a difficult task, both from the theoretical 
and computational point of view, and many questions remain open, as regards structural results and general algorithms. We refer the reader to Ernst Kleinert's survey \cite{MR1309127} for a general account on this topic.

In this paper we propose a method, meant to be very general, 
to compute a presentation of $\Lambda^{\times}$. This method is based on a combination of Vorono\"{\i} theory of perfect quadratic forms
 and Bass-Serre theory of graphs of groups. 
The latter is well-known to provide a very versatile tool
 for computing a presentation of a group $\Gamma $ acting on a connected graph 
$X$. In such a situation, a good knowledge of the quotient graph
 $\Gamma \backslash X$ yields virtually all the information on $\Gamma$. 
To be more precise, one has the following fundamental exact sequence (\cite[Theorem 3.6]{MR1239551})
\begin{equation}\label{bs} 
1 \longrightarrow\pi_1(X) \longrightarrow \pi_1(\Gamma \backslash\backslash X) \longrightarrow \Gamma \longrightarrow 1 
\end{equation}

where $\pi_1(X)$ (resp. $\pi_1(\Gamma \backslash\backslash X)$) denotes the fundamental group of the graph $X$
 (resp. of the quotient graph of groups $\Gamma \backslash\backslash X$), see \emph{loc. cit.} for precise definitions. From this exact sequence, one can derive, at least in principle, a presentation of the group $\Gamma$. Of course, to be of any  practical interest, the exact sequence (\ref{bs}) 
must be applied in a context where the fundamental groups $\pi_1(X)$ and $\pi_1(\Gamma \backslash\backslash X)$ are computable as easily as possible, without too much prior knowledge of the structure of $\Gamma$.

The graph on which we will let $\Gamma=\Lambda^{\times}$ act is built as a neighbouring graph of "perfect forms",
 where we use a suitable refinement of the original notion of perfect forms in \cite{Vo1}. 
The idea of using Bass-Serre theory in this context dates back to Soulé's paper \cite{MR0470141}, where this technique was applied to derive the explicit structure of  $SL_3(\ZZ)$ as an amalgamated product of small finite groups. 
Later, Opgenorth \cite{MR1881760} applied the same kind of ideas to 
compute the integral normalizer $\Gamma $ of 
a finite unimodular group, where he only used the surjectivity of 
the map $\pi_1(\Gamma \backslash\backslash X) \longrightarrow \Gamma$, which already allows 
for the computation of a
 generating set of $\Gamma$. 
Opgenorth's methods have been applied in \cite{ThesisHerbert} to compute generators for 
unit groups $\Gamma = \ZZ G^{\times }$ of integral group rings for small groups $G$.
Yasaki \cite{MR2721434} used similar ideas, combined with Macbeath theorem \cite{MR0160848} to obtain a presentation of some Bianchi groups. Here, we propose to use the full strength
 of  (\ref{bs}) to actually get a presentation of $\Gamma$. 
From the exact sequence (\ref{bs}), we see that this amounts essentially
 to computing the fundamental group of the neighbouring graph 
 of perfect forms, which is non-trivial in general 
(the graph is not a tree). There are several ways to do this.
 We choose to view the neighboring relation on perfect forms not only as a graph but as the
 $1$-skeleton of a CW-complex (the well-rounded retract $W$, see Section \ref{wr}).
 For groups acting on such complexes,
 there is a slightly refined version of Bass-Serre theory, 
due to K.-S. Brown (see \cite{MR739633}), 
which allows one to obtain a presentation of the group which 
involves only the $2$-skeleton of $W$. Beside the geometry of this complex, our algorithm only involves the computation of the stabilizers of some vertices or edges as well as "side-pairing" transformations. 
All these computations are performed using the Plesken-Souvignier algorithm, 
as implemented in MAGMA \cite{magma}. In particular, they essentially reduce to isometry testing of lattices and do not require any a priori knowledge of $\Lambda^{\times}$.

For sake of simplicity, we chose to develop the theory for simple algebras over $\Q$. 
The case of semisimple algebras requires only slight modifications.

In some special cases, \eg Bianchi groups or units of quaternion algebras, there are well-known methods based on an action of the relevant group on some \emph{hyperbolic space} (see \eg \cite{MR2073916}) for computing presentations. Our method applies in these cases too, and should be compared to the aforementioned ones, from which they differ essentially in that we use an action on a \emph{Euclidean space}.

The paper is structured as follows: in Sections 2 and 3, we define a certain space of "quadratic forms" acted on by the unit group $\Gamma=\Lambda^{\times}$ we want to study, and review some rather classical material about Vorono\"{\i}'s algorithm and the \emph{"well-rounded retract"} in this context. In Section 4 we explain how to use Bass-Serre theory to obtain a presentation of $\Gamma$. We also show in Section 5 how the previous idea can be used to solve the word problem in $\Gamma$. In Section 6, a selection of examples of applications of our method is presented. The final section provides an outline of the implementation of the algorithm. 

\section{Preliminaries} 
\subsection{Lattices}
Let $A=K^{n\times n}$ be a finitely generated simple algebra over $\QQ$,
 where  $K$ is a  skew field with center $k$ and $K^{n\times n}$ denotes the set of 
$n\times n$-matrices.  Let $V = K^n$ be the simple left $A$-module. Then
$K = \End_A(V)$ and we view $V$ as a right $K$-module.

Let $\Lambda$ be an order in $A$, and  $\Lambda^{\times}$ its group of units. 
We fix some left $\Lambda $-lattice $L$ in $V$ and let $\OO:= \End _{\Lambda }(L)$.
Then $\OO$ is an order in $K$ and $L$ is a right $\OO$-lattice. 
Put
$$\mathfrak M := \End_{\OO}(L)=\left\lbrace M \in K^{n\times n} \mid ML \subset L\right\rbrace.$$ 
If $\OO $ is a maximal order, then also $\mathfrak M$ is maximal, but for
arbitrary orders $\Lambda $ in division algebras $A=K$, one may always choose $L = \Lambda $ to 
achieve $\OO = \Lambda $ and $\Lambda = \mathfrak M  = \End _{\Lambda }(\Lambda )$.
In general $\Lambda \subseteq \mathfrak M $ is of finite index and also its unit group
$$\Lambda^{\times}=\Stab_{\mathfrak M ^{\times}}\Lambda = \{ a\in {\mathfrak M}^{\times } \mid a \Lambda = \Lambda \} .$$ 
has finite index in 
$$\mathfrak M^{\times } = \GL(L)=\left\lbrace a \in K^{n\times n} \mid aL = L\right\rbrace.$$ 
As the Vorono\"{\i}  algorithm is designed to work with endomorphism rings of lattices,
it is more efficient to compute 
 $\mathfrak M ^{\times}$ first and retrieve $\Lambda^{\times}$ by orbit stabiliser routines. 
Nevertheless we try to develop the theory,
 as much as possible, without the assumption that $\Lambda$ is  the endomorphism ring of a lattice.

\subsection{Forms} As explained in the introduction, we want to let the group $\Lambda^{\times}$ act on a space of "forms" associated to the algebra $A$. To that end, we first extend the scalars to $\RR$, and obtain a semi-simple real algebra
$$A_{\RR }\coloneqq   A \otimes _{\QQ } \RR=K_{\RR}^{n\times n}.$$

Let $d$ denote the degree of $K$, so $d^2 = \dim _{k} (K)$, and let
$$ \begin{array}{ll} 
\iota _1 ,\ldots , \iota _s  &   \mbox{ be 
the real places of $k$ that ramify in $K$,} \\
 \sigma _1,\ldots , \sigma _r &
\mbox{ the real places of $k$ that do not ramify in $K$ } \\
\tau _1,\ldots , \tau _t  & 
\mbox{ the complex places of $k$.} 
\end{array}
$$
Then
\begin{equation*}
K_{\RR }\coloneqq   K\otimes _{\QQ } \RR \cong 
\bigoplus _{i=1}^{s} \HH ^{d/2\times d/2} \oplus 
\bigoplus _{i=1}^{r} \RR^{d\times d} \oplus 
\bigoplus _{i=1}^{t} \CC^{d\times d}
\end{equation*}

and
\begin{equation}\label{iso} 
A_{\RR } \cong 
\bigoplus _{i=1}^{s}  \HH^{nd/2\times nd/2} \oplus 
\bigoplus _{i=1}^{r} \RR ^{nd\times nd} \oplus 
\bigoplus _{i=1}^{t}  \CC ^{nd\times nd}
\end{equation}
The ``canonical'' involution  $\ ^*$ (depending on the choice of this
isomorphism)
is defined on 
any simple summand of $K_{\RR }$
to be transposition for $\RR^{d\times d}$, transposition and complex 
(respectively quaternionic) conjugation for $\CC^{d\times d}$ and 
$\HH^{d/2\times d/2}$. 
The resulting involution  on $K _{\RR }$ is again denoted by $\ ^*$.
As usual it defines 
a mapping $\ ^{\dagger} : K_{\RR }^{m\times n} \to K_{\RR }^{n\times m}$
by applying $\ ^*$ to the entries and then 
transposing the $m\times n$-matrices. 
In particular this defines an involution $\ ^{\dagger} $ on 
$A _{\RR} = K_{\RR}^{n\times n} $.
Note that, in general, this  involution  will not fix $A\subset A_\RR$.

\begin{defn}\label{defSigma}
$\Sigma\coloneqq   \Sym (A _{\RR }) \coloneqq    \left\lbrace  F\in A_{\RR} \mid F^{\dagger } = F \right\rbrace  $ 
is the $\RR $-linear subspace of symmetric elements of $A _{\RR} $.
It supports the positive definite inner product 
$$\langle F_1,F_2 \rangle \coloneqq   \tr (F_1F_2 ) $$ 
where $\tr=\tr_{A _{\RR}/\RR}$ is the reduced trace of the semi-simple $\RR $-algebra $A _{\RR}$. Each element of $\Sigma$ can be identified, via (\ref{iso}), with a tuple $ (q_1,\ldots , q_s, f_1,\ldots , f_r, h_1,\ldots , h_t) $ of symmetric (resp hermitian) matrices. Then, one can define the open real cone $\mathcal P$ of positive elements in $\Sigma$ as
$$ {\mathcal P} = \Sym (A _{\RR }) _{>0}\coloneqq   \left\lbrace   (q_1,\ldots , q_s, f_1,\ldots , f_r, h_1,\ldots , h_t) \in \Sigma \mid q_i, f_j ,h_k \mbox{ pos. def.} 
 \right\rbrace  .$$
\end{defn}
The closure of $ {\mathcal P} $ in $\Sigma$ is denoted by $\overline{\mathcal P}$.

Recall that $V=K^{n\times 1}$ is the simple left $A$-module.
Any $F\in \Sigma  $ defines a 
quadratic form on $V_{\RR }$ by 
 $$F[x] \coloneqq   \langle F, x x^{\dagger } \rangle \in \RR  \mbox{ for all } x\in V_{\RR} .$$
This quadratic form is positive definite (resp. positive semidefinite) if and only if $F\in {\mathcal P}$ (resp. $F \in \overline{\mathcal P}$).

The group $\GL_n(K)$ acts on $\Sigma $ by 
\begin{equation}\label{act} 
(F,g) \mapsto g^{\dagger } F g  
\end{equation}
where we embed $A$ into $A_{\RR} $ to define the multiplication. This action preserves the cone $\mathcal P$. 
\subsection{Minimal vectors and Vorono\"{\i} domains}
As before we choose a left $\Lambda $-lattice $L$ in the
simple $A$-module $V=K^n$
and put $\mathfrak M := \End _{\OO} (L)$
(where $\OO := \End _{\Lambda }(L)$). 
Then $\mathfrak M$ is an order in $A$ that contains $\Lambda $ of finite index.

The $L$-minimum of a form  $F\in \mathcal P$ is defined as 
\begin{equation*} \min\nolimits _L (F) \coloneqq   \min_{\ell \in L-\left\lbrace 0\right\rbrace } F[\ell ]\end{equation*}
and the set of $L$-minimal vectors of $F$ as
\begin{equation*} S_L(F) \coloneqq   \left\{ \ell \in L  \mid  F[\ell] = \min\nolimits _{L}(F) \right\}.  \end{equation*}
The $L$-Vorono\"{\i} domain of $F$ (or simply Vorono\"{\i} domain of $F$, if there is no ambiguity on the underlying lattice $L$)  is defined as 
\begin{equation*}
D_{F}\coloneqq   \left\lbrace \sum_{x \in S_L(F) } \lambda_{x }x x^{\dagger } \, , \lambda_x \geq 0 \right\rbrace 
\subset \overline{\mathcal P},
\end{equation*}
the closed convex hull of the rays $\RR_{\geq 0}x x^{\dagger }$ as $x$ ranges over $S_L(F)$.

The Vorono\"{\i} polyhedron $\Omega$ is defined as the closed convex hull of the rays $\RR_{\geq 0}x x^{\dagger }$ as $x$ ranges over $V$.

\begin{defn}
A form $F\in {\mathcal P}$ is $L$-perfect (or simply perfect if there is no ambiguity on the underlying lattice) if one the following equivalent conditions holds
\begin{enumerate}
\item The forms $x x^{\dagger }$, where $x$ ranges over $S_L(F)$, span the whole space $\Sigma$.
\item The $L$-Vorono\"{\i} domain of $F$ has non-empty interior.
\end{enumerate}
\end{defn}

The Vorono\"{\i} domain $D_{F}$ of a perfect form $F \in \mathcal P$ 
is thus an $N$-dimensional polyhedral cone in the Euclidean space
 $\Sigma$, where $N=\dfrac{nd}{2}\left(nd\left[ k:\QQ\right] +r-s\right)$ is the dimension of $\Sigma$. It has finitely many facets, i.e. codimension $1$ faces. To each facet $\mathcal F$ one associates a \emph{direction} $H$, that is a normal facet vector,
 pointing towards the interior of $D_{F}$. In other words, $0\neq H \in \Sigma$ is a direction of $D_{F}$ if :
\begin{itemize}
\item  $\langle H, x x^{\dagger } \rangle \geq 0$ for all $x \in S_L(F) $,
\item the forms $x x^{\dagger }$, as $x$ ranges over the set of minimal vectors of $L$ such that $\langle H, x x^{\dagger } \rangle=0$, generate a hyperplane of $\Sigma$ .
\end{itemize}
The following lemma is at the core of Vorono\"{\i} theory :
 \begin{lemma}\label{deadend}  Let $F$ be a perfect form $F \in \mathcal P$, and $H$ a direction of its Vorono\"{\i} domain. Then there exists a well-defined positive real number $\lambda$ such that $F+\lambda H$ is perfect, and the Vorono\"{\i} domains of $F$ and $F+\lambda H$ share a common facet.
 \end{lemma}
 \begin{proof} This follows from \cite[Proposition 1.8]{MR1881760} and its proof. In particular, the real number $\lambda$ can be defined as 
\begin{equation}
\lambda =\sup  \left\lbrace \theta \in \RR_{>0} \mid F+\theta H \in \mathcal{P} \text{ and } \min\nolimits _L ( F+\theta H)= \min\nolimits _L (F)\right\rbrace.
\end{equation}
The only thing to check is that $\lambda< +\infty$, which amounts to proving that $H \notin \overline{\mathcal P}$ (see \cite[Proposition 13.1.8]{MR1957723} or the discussion following \cite[Proposition 1.8]{MR1881760}). Assume by way of contradiction that $H$ is positive. Then its kernel (the set of $x \in V_{\RR}$ such that $Hx=0$) coincides with its radical (the set of $x \in V_{\RR}$ such that $H\left[x\right]=0$), and contains the set $S(\mathcal F)$ of  minimal vectors $x \in L$ whose image $x x^{\dagger }$ in $\Sigma$ generate the facet $\mathcal F$ corresponding to $H$. So these vectors span a $K$-subspace of dimension at most $n-1$ of $K^n$, since otherwise $H$ would be zero, which means that there exists $0\neq y \in K^n$ such that 
\begin{equation}\label{hyper} 
y^{\dagger}x = 0 \text{ for all } x \in S(\mathcal F).
\end{equation}
This implies, in particular, that $n$ is at least $2$ (there is at least one element in the set of minimal vectors belonging to this facet). Finally, the matrices $x x^{\dagger }$, $ x \in S(\mathcal F)$, generate a subset of dimension at most $N-n$ of $\Sigma$ : indeed each of the $n$ columns of
\begin{equation*}
 x x^{\dagger }=\begin{pmatrix}
 x_1x_1^{*} &  x_1x_2^{*} & \cdots &  x_1x_n^{*} \\ 
  x_2x_1^{*} &  x_2x_2^{*} & \cdots &  x_2x_n^{*} \\ 
 \vdots & \vdots &  & \vdots \\ 
  x_nx_1^{*} &  x_nx_2^{*} & \cdots &  x_nx_n^{*}
 \end{pmatrix} 
\end{equation*} 
lies in the hyperplane determined by (\ref{hyper}). This yields a contradiction since $N-n <N-1$. 
\end{proof}

With the notation above, the form $F+\lambda H$ is called the \emph{neighbour} of $F$ in the direction $H$.

\subsection{Vorono\"{\i} algorithm} 
Roughly speaking, Vorono\"{\i} theory, or its variants, says that the Vorono\"{\i} polyhedron $\Omega$ may be tiled by the cones $D_{F}$ as $F$ ranges over the set of \emph{perfect} forms (see below for a more precise statement). There is also a dual formulation, in terms of minimal classes, which will provide the graph on which to apply Bass-Serre theory.

\begin{theorem}\label{vt} 
The $L$-Vorono\"{\i} domains of perfect forms constitute a locally finite exact tessellation of $\mathcal{P}$, that is :
\begin{enumerate}
\item $\mathcal{P}\subset \bigcup_{F \text{ perfect}} D_F$,
\item for any two perfect forms $F$ and $F'$ one has $\mathring{D_F} \cap D_{F'} \neq \emptyset$ if 
and only if $F=\lambda F'$ for some $\lambda \in \RR _{>0}$,
\item \label{neighb} each facet of the Vorono\"{\i} domain $D_F$ of a perfect form $F$ is a common facet of exactly two Vorono\"{\i} domains $D_F$ and $D_{F'}$ of perfect forms $F$ and $F'$,
\item \label{facets} the Vorono\"{\i} domain of a perfect form $F$ intersects only finitely many  Vorono\"{\i} domains of perfect forms.
\end{enumerate}
Moreover this tessellation is finite up to the action of  $\Lambda^{\times}$, in the following sense:
\begin{enumerate}[resume]
\item There are finitely many $\Lambda^{\times}$-inequivalent perfect forms of minimum $1$.
\end{enumerate}
 
\end{theorem}
\begin{proof} The proof of the first four assertions is a direct application of \cite[Theorem 1.9]{MR1881760} (with the terminology used there, one has to check that the image of $L$ in $V_{\RR}$ is a discrete admissible set, which is straightforward). The assertion regarding finiteness can be established using Mahler's compactness theorem and standard arguments from reduction theory. A quick alternative proof can be derived from results of Ash (\cite{MR747876}) as follows: First, since $\left[\GL(L):\Lambda^{\times}\right]$ is finite, it is enough to prove that there are finitely many  $\GL(L)$-inequivalent forms. Now the set $\mathcal V$ of $L$-perfect forms of minimum $1$ is clearly a discrete and closed subset of $ \mathcal P$. Moreover, it is contained in the set of well-rounded forms (see next section), whose quotient modulo $\GL(L)$ is compact (see \cite{MR747876} main theorem, section 2). The conclusion follows.
\end{proof}

The radical $\Rad(F)$ of a form $F \in \overline{\mathcal P}$ is the set of $x \in V_{\RR
}$ such that $F[x]=0$. We say that the radical of $F$ is defined over $K$ if there exists a $K$-subspace $W$ of $V$ such that $\Rad(F)= W \otimes_{\QQ} \RR$.  The rational closure $\overline{\mathcal{P}}^{K}$ of $\mathcal P$ is the set of forms in $\overline{\mathcal P}$, the radical of which is defined over $K$. The following corollary is a straightforward generalization of \cite[Proposition 36]{MR3074816} which was obtained under the restriction $d=1$. Our proof is slightly shorter, since the most difficult part (the fact that $\mathcal P$ is contained in $\Omega$)  is now a simple consequence of Theorem \ref{vt}.
\begin{corollary}\label{tesse}
The Vorono\"{\i} polyhedron $\Omega$ coincides with the rational closure of $\mathcal P$. The Vorono\"{\i} tessellation takes the final form
\begin{equation}\label{desc} 
\mathcal{P}\subset \bigcup_{F \text{ perfect}} D_F=\Omega=\overline{\mathcal{P}}^{K} \subset\overline{\mathcal{P}} .
\end{equation}
\end{corollary}
\begin{proof} The inclusions
$$\mathcal{P}\subset \bigcup_{F \text{ perfect}} D_F \subset \Omega$$ 
are clear (the first inclusion is a consequence of the previous theorem and the second is obvious from the definitions of $\Omega$ and $D_F$). 

It is also easy to see that $\Omega \subset \overline{\mathcal{P}}^{K}$: Indeed, for any nonzero $F \in \Omega$, there exist vectors $x_1$, \dots, $x_m$ in $V$ and a family of positive real numbers  $\lambda_1$, \dots, $\lambda_m$ such that $F=\sum_{i=1 } ^m \lambda_{i }x_i x_i^{\dagger }$. The radical of such an $F$ is the set of $y \in V_{\R}$ such that  $$0=\sum_{i=1 } ^m \lambda_{i }\langle x_i x_i^{\dagger }, y y^{\dagger } \rangle=\sum_{i=1 } ^m \lambda_{i }\tr \left(x_i^{\dagger } y y^{\dagger } x_i\right)= \sum_{i=1 } ^m \lambda_{i }\tr_{K_{\RR}/\RR}\left(\left(x_i^{\dagger }y\right)\left(x_i^{\dagger }y\right)^{*}\right)$$
which means that $x_i^{\dagger }y=0$ for all $i$, since the $\lambda_i$s are positive, and $\tr_{K_{\RR}/\RR}\left(aa^{*}\right) > 0$ for any nonzero $a \in K_{\RR}$. In other words, $\Rad(F)$ is the intersection of  $t$ hyperplanes in $V_{\R}$ which are clearly defined over $K$ since the $x_i$ are in $V$.

The reverse inclusion $\overline{\mathcal{P}}^{K} \subset \Omega$ can be established  using the same argument as \cite[Proposition 36]{MR3074816}. Consider a form $F \in \overline{\mathcal{P}}^{K}$.  Assuming that $A=K^{n\times n}$, there exists a $K$-subspace $W$ of $V=K^n$ of dimension $m \leq n$ such that $\Rad(F)= W \otimes_{\QQ} \RR$. Consequently, $F$ is $\GL_n(K)$-equivalent to a form of the shape 
$$ \begin{pmatrix}
0 & 0 \\ 
0 & F_W
\end{pmatrix} $$
with $F_W$ in $\mathcal{P}_W = \Sym (B_{\RR }) _{>0}$, where $B=K^{m\times m}\cong \End_{K}(W)$. As already seen, this cone $\mathcal{P}_W$ is contained in the corresponding Vorono\"{\i} cone, which means that there exists vectors $y_1$, \dots, $y_{\ell}$ in $K^m$  and a family of positive real numbers  $\lambda_1$, \dots, $\lambda_{\ell}$ such that $F_W=\sum_{i=1 } ^{\ell} \lambda_{i }y_i y_i^{\dagger }$. But then $F$ is $\GL_n(K)$-equivalent to $$\begin{pmatrix}
0 & 0 \\ 
0 & F_W
\end{pmatrix} =\sum_{i=1 } ^{\ell} \lambda_{i }\begin{pmatrix} 0 \\ y_i \end{pmatrix} \begin{pmatrix} 0 \\ y_i \end{pmatrix} ^{\dagger } \in \Omega,$$ whence the conclusion.

The final step to prove (\ref{desc}) is to show that $\Omega \subset \bigcup_{F \text{ perfect}} D_F$. Let $Q=\sum_{i=1 } ^m \lambda_{i }x_i x_i^{\dagger }$ be a  nonzero element in $\Omega$. We may assume, without loss of generality, that all $\lambda_i$s are $>0$ and that the $x_i$ are in $L \setminus\left\lbrace0\right\rbrace$. Let $F_0$ be a perfect form with $L$-minimum $1$. If $Q \notin D_{F_0}$, then there exists a direction $H$ of $D_{F_0}$ such that $\langle H, Q \rangle < 0$. 
Consequently, if $F_1=F_0+\lambda H$ is the neighbour of $F_0$ in the direction $H$, then 
\begin{equation}
\langle Q, F_1 \rangle = \langle Q, F_0 \rangle + \lambda \langle H, Q \rangle < \langle Q, F_0 \rangle .
\end{equation}
We can pursue this process as long as $Q$ is not found to belong to the Vorono\"{\i} domain of a perfect form, and build a sequence $\left(F_n\right)_{n \in \NN}$ of perfect forms in $\mathcal V$, the set of perfect forms with $L$-minimum $1$, such that  the sequence $\left(\langle Q, F_n \rangle\right)_{n \in \NN}$ is strictly decreasing. On the other hand, the sequence $\left(\langle Q, F_n \rangle\right)_{n \in \NN}$ is easily seen to assume only finitely many values. Indeed, we have
\begin{equation}\label{six} 
\langle Q, F_0 \rangle \geq \langle Q, F_n \rangle=  \sum_{i=1 } ^{\ell} \lambda_i F_n[x_i]
\end{equation} 
and we know that for every positive definite form $F$ and positive $\theta$, the set $F[\theta] = \left\lbrace F[x], x \in  L\right\rbrace \cap \left[0,\theta\right]$ is finite and depends on $F$ only up to $\GL(L)$-conjugacy. Since $\mathcal V \slash\GL(L)$ is finite, the right-hand side of (\ref{six}) can thus take only finitely many values. This shows that the process must terminate, and $Q$ belongs to the Vorono\"i domain of a perfect form.\end{proof}

\section{A $CW$-complex}\label{wr} 

The Vorono\"{i} tessellation of Theorem \ref{vt} yields a cellular decomposition of  $\mathcal P$. Dual to it, one has a natural $CW$-complex, acted on by $\Gamma=\Lambda ^{\times}$, carried by the set of \textit{well-rounded} forms (see definition below). This cell-complex has been studied by many authors, especially Avner Ash in \cite{MR747876}, to which we refer in what follows. 
\begin{defn}
A form $F \in \mathcal P$ is \emph{well-rounded} if its set of  minimal vectors   $S_L(F)$ contains a
$K$-basis of $V$.
\end{defn}
The cell structure on $\mathcal P$ is induced by the decomposition into minimal classes, which are defined as follows :
\begin{defn}
Two elements $F_1$ and $F_2 \in {\mathcal P}$ are called
\emph{minimally equivalent with respect to $L$}, if $S_L(F_1) = S_L(F_2)$. We denote by $\Cl_L (F) := \left\lbrace  H \in {\mathcal P} \mid S_L(H) = S_L(F) \right\rbrace$ the \emph{minimal class} of $F$.
If $C=\Cl _L (F)$ is a minimal class then we define
$S_L(C) = S_L(F)$ the associated set of minimal vectors.
A minimal class $C=\Cl _L (F)$ is called \emph{well rounded} if the form $F$ is.
\end{defn}


One has the following equivalent characterizations of well-rounded forms (resp. classes): 
\begin{lemma}\label{cwr}  Let $\mathring{D_F}$ denote the relative interior of the Vorono\"{\i} domain of a form $F$ (\ie $D_F$ deprived of its proper faces). Then, the following assertions are equivalent 
\begin{enumerate}
\item\label{un}  $F \in \mathcal P$ is well-rounded,
\item\label{deux}  $\mathring{D_F} \cap \mathcal P \neq \emptyset$,
\item\label{trois}  $D_F \not\subset \partial \mathcal P$.
\end{enumerate}
\end{lemma}
\begin{proof}
Assume that $F$ is well-rounded, and consider the form $H=\sum_{x \in S_L(F)} x x^{\dagger } \in \mathring{D_F}$. For any $y \in \Rad(H)$ one has
$$0=H\left[y\right]=\sum_{x \in S_L(F)}\langle x x^{\dagger }, y y^{\dagger } \rangle=\sum_{x \in S_L(F)}\tr \left(x^{\dagger } y y^{\dagger } x\right)=\sum_{x \in S_L(F)}\tr_{K_{\RR}/\RR}\left(\left(x^{\dagger }y\right)\left(x^{\dagger }y\right)^{*}\right)$$
whence $x^{\dagger }y=0$ for all $x \in S_L(F)$, hence $y=0$ since $S_L(F)$ spans $K^n$. Thus $H \in \mathring{D_F} \cap \mathcal P$ which shows that $(\ref{un})\Rightarrow(\ref{deux})$. The implication $(\ref{deux})\Rightarrow(\ref{trois})$ is obvious. As for $(\ref{trois})\Rightarrow(\ref{un})$, we note that if $F$ is not well-rounded, then one can find a non zero $y \in K^n$ such that $y^{\dagger}x=0$ for all $x \in S_L(F)$, whence we deduce that $xx^{\dagger}\left[y\right]=0$ for all $x \in S_L(F)$, which implies that $y$ belongs to the radical of every $H \in D_F$. Thus $ D_F\subset \partial \mathcal P$.
\end{proof}

The action (\ref{act}) of  $\GL_n(K)$ on $\Sigma $, restricted to its subgroups  $\Lambda^{\times}\subset\mathfrak M^{\times}=\GL(L)$, induces an action on the set of minimal classes.

Clearly, because of positive definiteness, the stabilizer $\Stab_{\Lambda^{\times}}(F):=  \left\lbrace  g \in \Lambda ^{\times} \mid g^{\dagger} F g = F \right\rbrace$ is always a finite subgroup of $\Lambda^{\times}$. 
We can define similarly the stabilizer
of a minimal class as
$$\Stab _{\Lambda^{\times}} (C) = \left\lbrace  g\in {\Lambda^{\times}} \mid g S_L(C) = S_L(C) \right\rbrace.$$
\begin{lemma}
The stabilizer $\Stab _{\Lambda^{\times}} (C)$ of a well-rounded class is finite.
\end{lemma}
\begin{proof}
It follows from \cite[Lemma 5.3]{CN} that $\Stab _{\Lambda ^{\times}} (C)=\Stab _{\Lambda ^{\times}} (T_C^{-1})$ where
$ T_C:= \sum _{x\in S_L(C)} x x^{\dagger } \in {\mathcal P} $ is the canonical form 
associated to $C$.
So 
$\Stab _{\Lambda ^{\times}} (C)$ is the stabilizer of some positive form and therefore
a finite group.
\end{proof}

Positive real homotheties preserve minimal equivalence and the set of well-rounded forms. The quotient $\widetilde{W}$ of $W$ by these homotheties inherits a well-defined $CW$-complex structure with the following properties.

\begin{theorem}[\cite{MR747876}]
Let $W$ be the set of well-rounded forms in $\mathcal P$, and $\widetilde{W}=\RR _{>0}\backslash W$.
Then the correspondence
$$\Cl _L (F) \longleftrightarrow \mathring{D_F} \cap \mathcal P $$
is an inclusion-reversing bijection between the set of cells  of $\widetilde{W}$, \ie \emph{minimal classes}, and the set of open Vorono\"{\i} domains not contained in the boundary of $\mathcal P$.  In particular, $0$-cells of $\widetilde{W}$ correspond to perfect forms in this bijection. The group $\Lambda^{\times}$ acts cellularly on $\widetilde{W}$, and the cells have finite stabilizers.
\end{theorem}
\begin{proof} If $\Lambda$ is a maximal order, 
this is the main theorem of \cite{MR747876}. The general case follows easily, since $\Lambda^{\times}$ is a finite index subgroup of the unit group of any of its maximal overorders.
 See also P. Gunnells' appendix to the book \cite{MR2289048} for a nice explanation of the duality between the Vorono\"{\i} complex and the well-rounded retract, together with their cell decompositions.
\end{proof}

\begin{rem}
Scaling invariance allows to work within the set $\widetilde{\mathcal P}=\RR _{>0}\backslash \mathcal P$, which we can identify with the set $\left\lbrace F \in \mathcal P \mid \min_L(F)=1\right\rbrace$. Using a classical terminology, $\widetilde{\mathcal P}$ can thus be viewed as the boundary of a \emph{Ryshkov polyhedron} (\cite{MR0276873}), which is locally finite (see \cite{MR2537111,MR2466406}). In particular, $\widetilde{\mathcal P}$ is a piecewise linear hypersurface, whose faces are the minimal classes. The bounded faces correspond to well-rounded classes and actually are polytopes.
\end{rem}

\section{Bass-Serre theory}\label{BassSerre}
In this section, we describe the theory underlying our algorithm for computing $\Gamma=\Lambda^{\times}$. Almost all the material here is borrowed, with hardly any change, from Brown's paper \cite{MR739633}. 

Let $\widetilde{W}$ be equipped with its cell structure, as in the previous section. We denote by $\widetilde{W}_i$ its $i$-skeleton. We can see $\mathcal G \coloneqq \widetilde{W}_1$ as a \emph{graph} in the sense of  \cite{MR0476875}, with vertex set $\mathcal V \coloneqq \widetilde{W}_0$
 and edge set $\mathcal E$ consisting of $1$-cells together with an orientation. 
Each edge $e$ has an origin $o(e)$ and a terminus $t(e)$, corresponding to its orientation. For each such $e$, we define $\overline{e}$ as the same $1$-cell, together with the reversed orientation ($o(\overline{e})=t(e)$ and $t(\overline{e})=o(e)$).

If there exists $g \in \Gamma$ such that $g(e)=\overline{e}$, one says that the edge $e$ is inverted under the action of $\Gamma$. One technical difficulty when applying Bass-Serre theory in its original form is precisely that the definition of a graph adopted either in \cite{MR0476875} or in \cite{MR1239551} forbids action of groups reversing the orientation of edges, a condition which is not necessarily satisfied in practice. One can easily get around this problem e.g. using barycentric subdivision. Brown's paper deals with this in a slightly different way, although essentially equivalently, which we summarize in the following steps:
 
\begin{enumerate}
\item  Choose an orientation on $1$-cells, in such a way that the orientation of those that are not inverted by the action of $\Gamma$ is preserved by this action. 
\item Split the set of edges $\mathcal E$  into a disjoint union $\mathcal E= \mathcal E^{+}\sqcup \mathcal E^{-}$, where $\mathcal E^{+}$
 denotes the set of edges  which are not inverted  under the action of $\Gamma$,
 and $\mathcal E^{-}$ its complement. 
For $e \in \mathcal E$ let $\Gamma_{\left\lbrace e,\overline{e}\right\rbrace}$ be the stabilizer
 of the set $\left\lbrace e,\overline{e}\right\rbrace$ and $\Gamma_e$ the stabilizer of $e$ 
(together with its orientation). 
Clearly, $\Gamma_e$ is a subgroup of $\Gamma_{\left\lbrace e,\overline{e}\right\rbrace}$, one has 
$\Gamma _e = \Gamma _{o(e)} \cap \Gamma _{t(e)} $ and
$$\left(\Gamma_{\left\lbrace e,\overline{e}\right\rbrace} : \Gamma_e \right) = \begin{cases}
1 \text{ if } e \in \mathcal E^{+}\\
2 \text{ if } e \in \mathcal E^{-}.
\end{cases} $$
\item Fix a tree $T$ of representatives of $\widetilde{W}_1 ~ \mathrm{mod} ~ \Gamma$, that is a sub-tree such that the set $V_T$ of its vertices is a set of representatives of  $\widetilde{W}_0 ~ \mathrm{mod} ~ \Gamma$, with the further assumption that all its edges are in $\mathcal E^{+}$. This implies in particular that its edges are pairwise inequivalent $\mathrm{mod} ~ \Gamma$.
\item Choose a set $E^{+}$ of representatives of $\mathcal E^{+} ~ \mathrm{mod} ~ \Gamma$ such that $o(e) \in V_T$ for all $e \in E^{+}$, and a set $E^{-}$ of representatives of $\mathcal E^{-} ~\mathrm{mod} ~ \Gamma$ such that $o(e) \in V_T$ for all $e \in E^{-}$.
\item \label{cinq} For every $e\in E^{+}$, choose $g_e \in \Gamma$ such that $g_e^{-1}(t(e)) \in V_T$, with the convention that $g_e=1$ whenever $e$ is an edge of $T$.
\item \label{cinqm} For every $e\in E^{-}$, choose 
$g_e \in \Gamma_{\{ e,\overline{e} \}} \setminus \Gamma _e$. 
\item  Choose a set $F$ of representatives of the $2$-cells of $\widetilde{W} ~ \mathrm{mod} ~ \Gamma$, and attach to every $2$-cell $\tau$ in $F$ a \emph{combinatorial path} $\alpha$, \ie a sequence $\left(e_1, e_2, \dots, e_m\right)$ of edges such that:

\begin{itemize}
\item $\partial \tau =\cup_i \sigma_i$, where $\sigma_i$ denotes the $1$-cell underlying $e_i$,
\item $v_0 \coloneqq o(e_1)$ is in $V_T$,

\item $t(e_i)=o(e_{i+1})$ for $1\leq i \leq m-1$ and $t(e_m)=o(e_1)$,
\item $e_{i+1} \neq \overline{e_i}$ for  $1\leq i \leq m-1$ and $e_{1} \neq \overline{e_m}$.
\end{itemize}
To each edge of this path, one attaches (non-canonically) an element $g_i$ of the subgroup generated by the various isotropy groups $\Gamma_v$ ($v \in V_T$), $\Gamma_{\left\lbrace e,\overline{e}\right\rbrace}$ ($e \in E$) and the $g_e$ ($e \in E^{+}$) chosen in step \ref{cinq}, such that the successive vertices belong to $V_T$, $g_1V_T$, $g_1g_2V_T$, \dots , $g_1g_2 \cdots g_m V_T$ (see \cite[Section 1]{MR739633} for a precise description of $g_i$). In particular, one has $g_1g_2 \cdots g_m \in \Gamma_{v_0}$. We call the sequence $\left(g_1, \cdots, g_m\right)$ the \emph{cycle associated to $\tau$}, and occasionally identify a cycle with the corresponding cell, when no confusion can ensue.
\end{enumerate}

Altogether, the previous data lead to the following presentation of $\Gamma=\Lambda^{\times}$:
\begin{theorem}[\cite{MR739633} Theorem 1] \label{maintheorem}
Let $\Gamma=\Lambda^{\times}$ be the unit group of an order in a finitely generated simple algebra over $\QQ$. Let $W$ be the set of well-rounded forms in $\mathcal P$, and $T$, $E=E^{+}\cup E^-$, $F$ be chosen as above. Then $\Gamma$ has the following structure:

\begin{equation}\label{amalg} 
\Gamma=\left(\Ast\limits_{v \in V_T}\Gamma_v\right)\ast \left(\Ast_{e \in E^-}\Gamma_{_{\left\lbrace e,\overline{e}\right\rbrace}}\right)\ast F(E^{+}) \slash R
\end{equation} where $\ast$ stands for the free product, $F(E^{+})$ denotes the free product on the set $\left\lbrace g_e, e \in E^+\right\rbrace$ and $R$ is the normal subgroup generated by:
\begin{itemize}
\item $g_e$, $e \in T$,
\item $g_e^{-1} \cdot g \cdot g_e \left(g_e^{-1} g  g_e \right)^{-1}$, $e \in E^{+}$, $g \in \Gamma_e \subset \Gamma _{o(e)}$,
\item\label{cycles}  $g_1 \cdot g_2\cdot \dots \cdot g_{m-1}\cdot g_m\left(g_1 \cdots g_m\right)^{-1}$, $\left(g_1, \cdots, g_m\right) \in F$
\end{itemize}

In other words, $\Gamma$ is generated by the subgroups
 $\Gamma_v$  ($v \in V_T$) and the elements $g_e$ ($e \in E^{+}\cup E^-$),
 subject to the following relations:
\begin{itemize}
\item[0.] The multiplication table of $\Gamma _v$ ($v\in V_T$). 
\item[1.] $g_e=1$ if $e$ is an edge of $T$.
\item[2.] $g_e^{-1} \cdot g \cdot g_e = g_e^{-1} g  g_e \in \Gamma _{w(e)}$,
 for $e \in E^{+}$ and $g \in \Gamma_e \subset \Gamma _{o(e)}$.
 \item[3.] $g_e\cdot g = g_eg$ and $g \cdot g_e=gg_e \in \Gamma _{o(e)}$,
 for $e \in E^{-}$ and $g \in \Gamma_e \subset \Gamma _{o(e)}$.
\item[4.] $g_1 \cdot g_2$ \dots $g_{m-1}\cdot g_m=g_1 \cdots g_m$ for any cycle $\left(g_1, \cdots, g_m\right)$ associated to a $2$-cell $\tau$.
\end{itemize}
\end{theorem}
\begin{proof} The description of $\Gamma$ in the first part of the theorem is \cite[theorem 1']{MR739633}, applied to $\widetilde{W}$ (which is contractible, hence simply connected). More precisely, the free product $$\left(\Ast\limits_{v \in V_T}\Gamma_v\right)\ast \left(\Ast_{e \in E^-}\Gamma_{_{\left\lbrace e,\overline{e}\right\rbrace}}\right)\ast F(E^{+})$$ modulo the normal subgroup generated by $$g_e, e \in T$$ and $$g_e^{-1} \cdot g \cdot g_e \left(g_e^{-1} g  g_e \right)^{-1}, e \in E^{+}, g \in \Gamma_e$$ is precisely the fundamental group of the barycentric subdivision $\mathcal G '$ of $\widetilde{W}_1$ acted on by $\Gamma$ (see the discussion preceding \cite[theorem 1']{MR739633}). Finally, one has to mod out by the fundamental group of $\mathcal G '$, which is the normal closure of the cycles associated to the $2$-cells in $F$. The second part of the theorem is just a rephrasing in terms of generators and relations.
\end{proof}

\begin{rem}\label{sidetransformation}
Let $\tilde{E}: = 
\{ e\in {\mathcal E} : o(e)\in V_T , t(e) \not\in V_T\} $  be the set of 
edges coming out of $T$. 
For any $e\in {\tilde{E}}$ 
there is some element $f\in E$ with $o(f) = o(e)$. Choose  some $h\in \Gamma _{o(e)}$ 
such that $h(f) = e$ (with the convention that $h=1$ if $e=f$). 
The element $g_e = h g_f \in \Gamma $ then satisfies $g_e^{-1}(t(e)) \in V_T$.
and is called the
{\bf side-transformation} corresponding to $e$. 
\end{rem}

\begin{rem}
In practical applications one usually does not fix an orientation 
on the edges of the graph $X$. 
The additional relations that we then need 
correspond to the so called ``side-pairings'' from the Poincar\'e-algorithm.
Let $e\in E$ and $g_e$ be as above, so that $g_e^{-1}(t(e)) = o(f) \in V_T$. 
We then call the edges $e$ and $g_e^{-1}(e) = \overline{f}$ {\em paired}.
Then $g_e (t(f)) = o(e) \in V_T$ because $\Gamma $ acts on the graph and
hence preserves edges. 
Also $g_f^{-1}(t(f)) \in V_T $ so these two vertices are in $V_T$ and
equivalent under $\Gamma $, which implies that 
$g_f^{-1}(t(f)) = o(e) \in V_T $ and $g_e^{-1} g_f \in \Gamma _v$. 
So if we do not choose an orientation and so do not 
restrict to those transformations $g_e$ with $o(e) \in V_T$ 
we need to add these additional relations. 
\end{rem}

\section{Solving the word problem} 

To solve the word problem we return to the tessellation by
Vorono\"{\i} domains $D_F$ of the perfect forms. 
According to Corollary \ref{tesse}  
this yields a locally finite 
exact tessellation of the Vorono\"{\i} polyhedron $\Omega$
that contains the open cone ${\mathcal P}$. 
Recall that ${\mathcal P}$ is a cone in the
Euclidean space $\Sigma $ from Definition \ref{defSigma}. 
Instead of working in projective space we take an
affine section of ${\mathcal P}$,
$${\mathcal P}' := \{ F\in {\mathcal P} \mid \tr(F) = 1 \} .$$

\begin{lemma} \label{geodesic}
For $x,y\in {\mathcal P}'$ let ${\mathcal G}:=\{ x+ s (x-y) \mid 
s\in [0,1] \}  \subset {\mathcal P}'$ the Euclidean geodesic
joining $x$ and $y$. 
Then  the set of all Vorono\"{\i} domains of perfect forms that 
meet ${\mathcal G}$,
$${\mathcal V}(x,y):= \{ D_F \mid D_F \cap {\mathcal G} \neq \emptyset \}, $$ 
is finite.
We call $d(x,y):=|{\mathcal V}(x,y)|$ the perfect-distance between $x$ and $y$. 
Note that this distance is $\Gamma $-invariant,
$d(x,y) = d(g(x),g(y)) $ for all $g\in \Gamma $.
\end{lemma} 

\begin{proof}
This is a general compactness argument: 
For any $z\in {\mathcal G}$ we may choose an open neighborhood $U_z$ of 
$z$ which intersects only finitely many Vorono\"{\i} domains $D_F$,
because of the local finiteness of the tessellation. 
As ${\mathcal G} \subset \cup _{z\in {\mathcal G}} U_z$ 
is an open covering of the compact set  ${\mathcal G}$ there 
is a finite subset $Z\subseteq {\mathcal G}$ such that 
$${\mathcal G} \subset \bigcup _{z\in Z} U_z \subset \bigcup _{z\in Z} 
\bigcup _{D_F\cap U_z \neq \emptyset} D_F . $$
\end{proof}

\begin{theorem}
There is an algorithm to express a given $g\in \Lambda ^{\times } $
as a word in the generators given in Theorem \ref{maintheorem}.
\end{theorem}

\begin{proof}
For the proof we give an algorithm to find such a word in 
the set of all side-transformations as defined in 
Remark \ref{sidetransformation} followed by some element in $\cup _{v\in V_T} \Gamma _v $.

Let 
$$F_T := \bigcup _{v\in V_T} D_v \cap {\mathcal P}' $$
be the union of all Vorono\"{\i} domains of the perfect forms
corresponding to the vertices of the tree $T$ chosen in Section \ref{BassSerre}.
Choose some inner point $x \in D_v \subset F_T$ and let $y:=g(x)$ be its
image under $g$ and ${\mathcal G}$ be the geodesic between $x$ and $y$. 
If $y\in D_w$ for some $w\in V_T$ then $w = v$ and 
$g\in \Gamma _v$. 
Otherwise
 ${\mathcal G}$ meets the boundary of $F_T$ in some point $p \in {\mathcal P}'$.
Let $$M := \{ D_F : F\in {\mathcal V} \setminus V_T , p \in D_F \}.$$
By Lemma \ref{geodesic} the set
 $M$ is finite. If $p$ is in the relative interior of some codimension 1
facet
(which will be almost always the case) 
then $M = \{ t(e) \} $ for some $e\in {\tilde{E}}$. 
Note that the perfect-distance $d(z,g(x))$ 
between any inner point $z$ of $D_{t(e)}$ which lies on $\mathcal{G}$
and $y$ is strictly smaller than $d(x,g(x))$.
Let $g_e$ be the corresponding
side-transformation  defined in Remark \ref{sidetransformation}
Then $g_e^{-1} (t(e)) = w\in V_T$ and $g_e^{-1}(z) \in D_w$ with
$$d(g_e^{-1} (z), g_e^{-1} (g(x)) ) = d(z,g(x)) < d(x, g(x) )$$
Replace $g$ by $g_e^{-1} g$, $x$ by $g_e^{-1} (z)$ and 
$y$ by $g_e^{-1}(y)$ and continue. 


In the unlikely event that $M$ contains more than one element (i.e. $\mathcal{G}$ meets the intersection of at least two distinct facets containing $p$) one chooses a different starting point $x'$ in a small neighborhood of $x$ which is still contained in $D_v$ while keeping the endpoint $y$ of $\mathcal{G}$ fixed. Since the intersection of two distinct facets is of codimension at least $2$ it is (from a measure theoretic point of view) highly unlikely that $\mathcal{G}$ again meets more than one facet of the fundamental domain. Hence after a small number of modifications of $x$ and $\mathcal{G}$ we are in the situation described above.
\end{proof}

\section{Examples} 

In this section we list a few examples to illustrate the algorithm.
Many more examples can be found in a database for unit groups of orders 
linked to the authors' homepages, where one will also find Magma implementations 
of the algorithms. 

\subsection{The rational quaternion algebra ramified at 2 and 3}
To illustrate the theory of the previous sections we comment on 
a very easy example. 
Take the rational quaternion algebra ramified at 2 and 3,
$${\mathcal Q}_{2,3} = \left( \frac{2,3}{\Q} \right) = 
 \langle i,j \mid i^2=2, j^2 = 3, ij=-ji \rangle 
= \langle \diag (\sqrt{2},-\sqrt{2}), \left(\begin{array}{cc} 0 & 1 \\ 3 & 0 \end{array} \right) \rangle . $$
Then a maximal order is 
$\Lambda = \langle 1,i,\frac{1}{2} (1+i+ij) , \frac{1}{2} (j+ij) \rangle $.
So here
$V=A={\mathcal Q}_{2,3} $, $A_{\R} = \R^{2\times 2}$, $L=\Lambda $.
If we embed 
 $A$ into $A_{\R }$ using the maximal subfield $\Q[\sqrt{2}]$
we find three perfect forms representing the $\Lambda ^{\times }$-orbits 
on the set of all perfect forms:
$$ 
F_1 = \left( \begin{array}{@{}c@{}c@{}} 
         1 & 2-\sqrt{2}  \\
    2-\sqrt{2} &          1 
\end{array}\right) , ~
F_2 = \left( \begin{array}{@{}c@{}c@{}} 
    6-3\sqrt{2}  &            2\\
               2  &  2+\sqrt{2}
\end{array}\right), ~
F_3 = \diag( -3\sqrt{2} + 9, 3\sqrt{2} + 5) 
 $$
with stabilizers 
$$\Stab _{\Lambda ^{\times }} (F_1) = \langle -1 \rangle,
\Stab _{\Lambda ^{\times }} (F_2) = \langle \beta  \rangle \cong C_4,
\Stab _{\Lambda ^{\times }} (F_3) = \langle \alpha \rangle \cong C_6 . $$


The tessellation  for ${\mathcal Q}_{2,3} \hookrightarrow 
\Q [\sqrt{2}]^{2\times 2}$
and the relevant part of the resulting graph (dual to the
tessellation) is visualised in the following pictures.
\\
\includegraphics[scale=0.4]{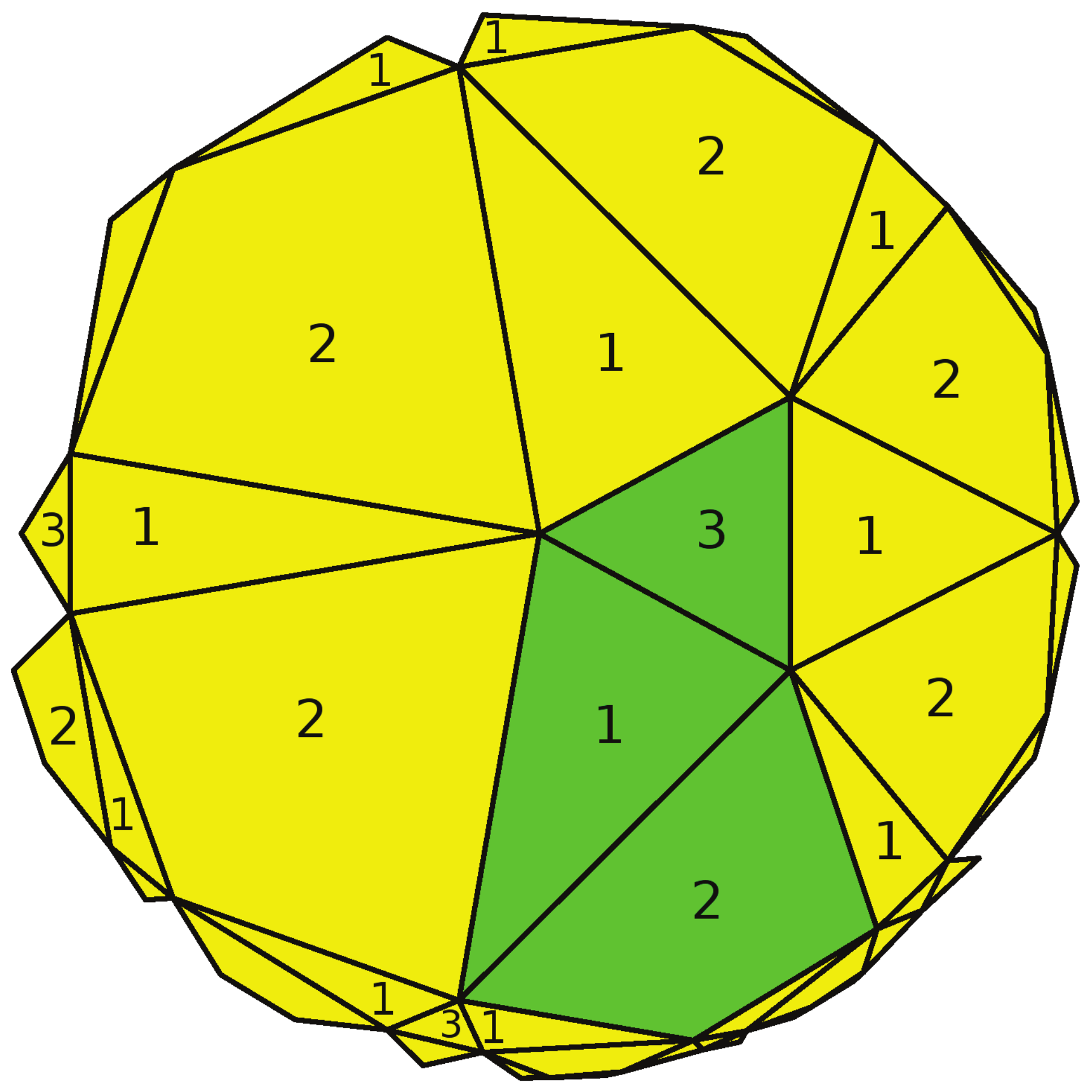}
\hfill{\includegraphics[scale=0.9]{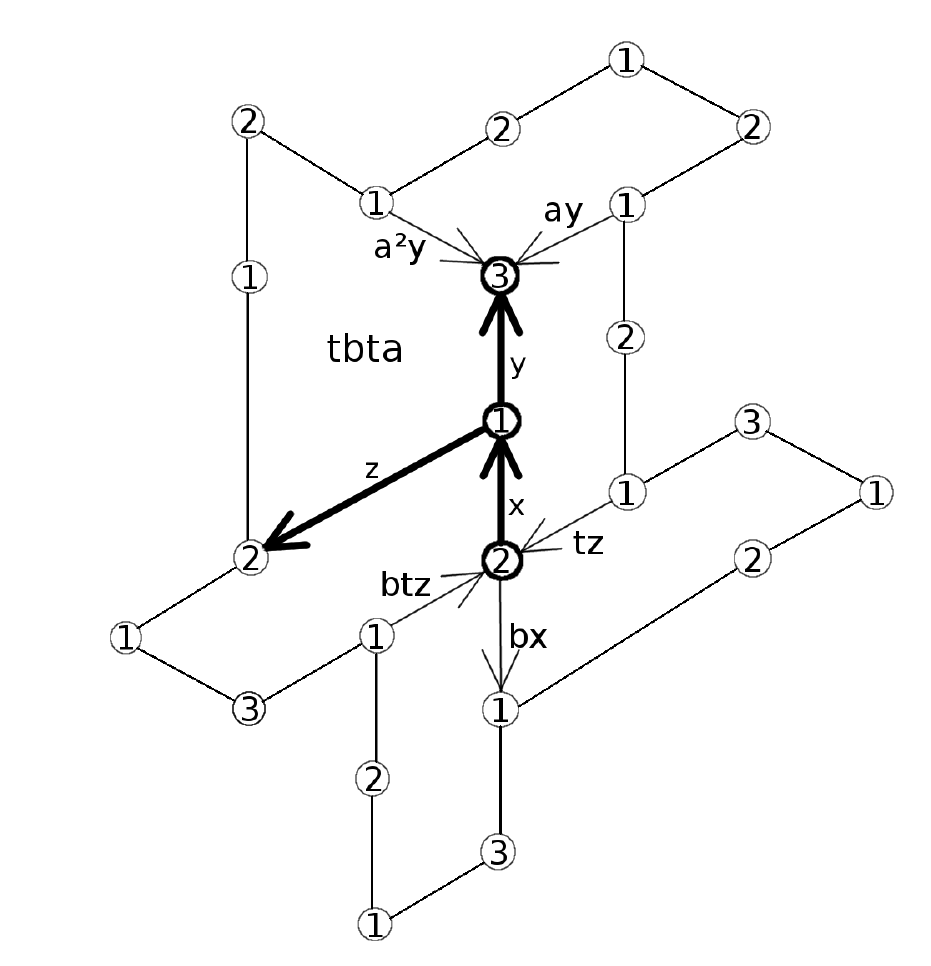}}

So $V=\{1,2,3\}$ is represented by the double circled vertices, 
and $E=E^+ = \{ x,y,z \}$. 
We have $\Gamma _1 = 1$, $\Gamma _2 = \langle b \rangle \cong C_2$ and
$\Gamma _3 = \langle a \rangle \cong C_3$ and put 
$g_z =: t$. 
Then 
$$\Gamma = \Lambda ^{\times }/\langle \pm 1 \rangle = \langle a,b,t \mid a^3, b^2, atbt \rangle .$$
Note that all cycle relations are conjugate as there is just 
one $\Lambda ^{\times }$-orbit on the minimal classes of perfection 
corank 2. 

To illustrate that the resulting tessellation depends on the 
chosen maximal subfield we redo the computations for the
maximal subfield $\Q [\sqrt{3}]$ instead of $\Q [\sqrt{2}]$. 

The tessellation for ${\mathcal Q}_{2,3} \hookrightarrow 
\Q [\sqrt{3}]^{2\times 2}$ is as follows. \\
\includegraphics[scale=0.4]{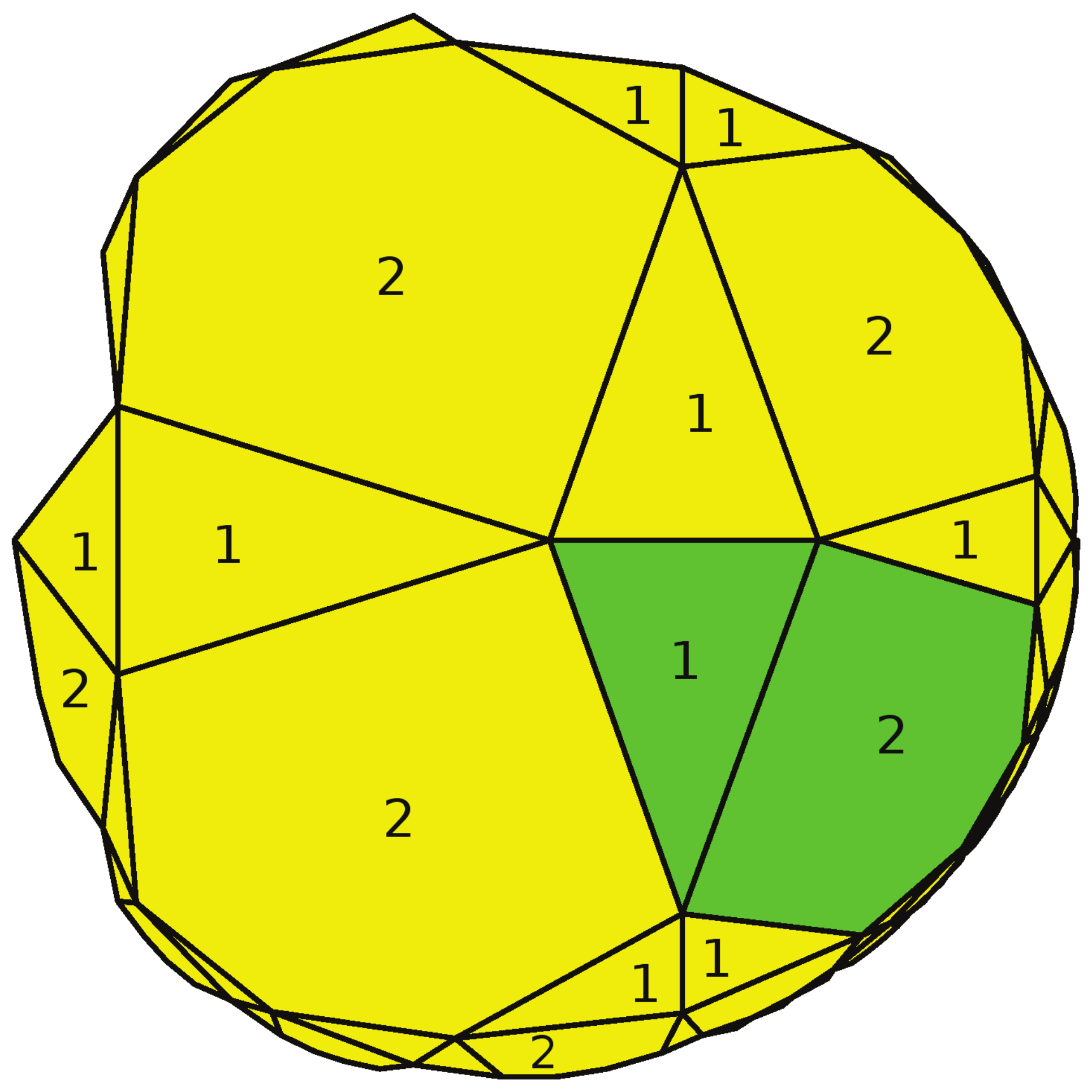}
\hfill{\includegraphics[scale=1.0]{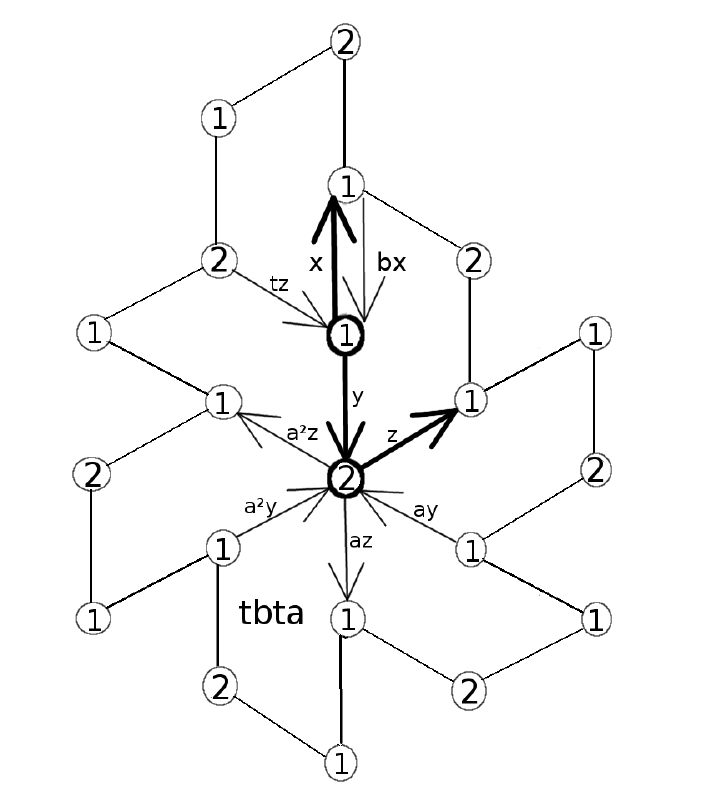}}

So here we obtain only 2 perfect forms, $f_1$ and $f_2$, say,
with $\Stab _{\Lambda ^{\times }}(f_1) = \langle -1 \rangle $, 
$\Stab _{\Lambda ^{\times }} (f_2) = \langle \alpha \rangle \cong C_6$. 

\subsection{The rational quaternion algebra ramified at 19 and 37}

This example illustrates the power of Vorono\"{\i}'s algorithm.
Let $\Lambda $ be a maximal order in the 
quaternion algebra $\left( \frac{19,37}{\Q } \right) $,
the rational quaternion algebra ramified at 19 and 37. 
The Fuchsian Group package in Magma does not return a 
presentation of the unit group $\Lambda ^{\times }$ in a reasonable 
time. Our algorithm takes about 5 minutes to compute 
$\Lambda ^{\times }/\langle \pm 1 \rangle = \langle H_1,\ldots , H_{56} \rangle$ 
with the single relator 
$$
\begin{array}{l} 
H_{21}  H_{10}^{-1}  H_{44}^{-1}  H_{49}^{-1}  H_{14}^{-1} 
 H_{55}^{2}  H_{4}  H_{42} H_{15}^{-1}  H_{46} 
 H_{19}^{-1}  H_{52}^{-1}  H_{20}  H_{17}^{-1}  H_{9} 
 H_{54}  H_{39} H_{16}^{-1}  H_{48}  H_{3}^{-1}  H_{44}^{-1}  H_{38} 
\\
 H_{2}  H_{26}^{-1}  H_{35}^{-1} H_{18}  H_{12}  
H_{56}^{2}  H_{1}  H_{20}  H_{25}  H_{24} 
 H_{23}  H_{5}  H_{50} H_{8}^{-1}  H_{41} 
 H_{35}^{-1}  H_{2}  H_{15}  H_{28}^{-1}  H_{5}  H_{43}  H_{53}^{-1}
\\
    H_{1}^{-1}  H_{34}^{-1}  H_{52}  H_{49}^{-1}  H_{48} 
 H_{8}^{-1}  H_{33}^{-1}  H_{14} H_{3}  H_{27}  
H_{36}^{-1}  H_{40}^{-1}  H_{47}  H_{9}^{-1}  H_{22} 
 H_{13}^{-1} H_{53}^{-1 } H_{39}  H_{27}  H_{51}^{-1}  H_{13}  H_{46} 
\\
 H_{47}^{-1}  H_{43}^{-1} H_{17}^{-1}  H_{37}^{-1}  H_{40}^{-1} 
 H_{21}  H_{30}  H_{6}  H_{12}^{-1}  H_{32}
    H_{54}^{-1}  H_{28}^{-1 } H_{36}  H_{22}  H_{29}^{-1} 
 H_{7}^{-1}  H_{45}^{-1}  H_{26}^{-1} H_{50}^{-1}  H_{32}^{-1}  H_{11}^{-1}
\\
  H_{51}^{-1}  H_{30}  H_{18}  H_{29}  H_{16}^{-1}
     H_{33}  H_{34}  H_{41}  H_{11}  H_{7}  H_{37}^{-1} 
 H_{42}^{-1}  H_{10}^{-1} H_{23}^{-1}  H_{6}  H_{31}^{-1}  H_{45}^{-1} 
 H_{19}^{-1}  H_{4}  H_{25}  H_{31}  H_{38} H_{24}.
\end{array} 
$$

\begin{rem}
One might want to compare this result with the well known formula for 
the genus $g$ of the associated Shimura curve (see e.g. \cite{Shimizu}).
Note that $2g = \dim (\Hom (\Gamma , \CC ) ) $ where $\Gamma = \{ 
g\in \Lambda ^{\times } \mid \nred(g) = 1\} /\langle -1 \rangle $. 
From our presentation we obtain that $\Gamma /\Gamma ' \cong \ZZ^{110}$, 
from which we get $110=2g = (19-1)(37-1)/6+2 $ as predicted. 
\end{rem}

\subsection{Quaternion algebras over imaginary quadratic fields}

The number of perfect forms and hence the performance of our algorithm
depends on the choice of the involution $\dagger $ on $A_{\RR }$. 
For certain quaternion algebras $A$, there is a canonical way to choose such an involution $\dagger $ that 
preserves $A$:

\begin{rem} 
Let $k$ be a CM-field with complex conjugation $\overline{\phantom{a}}$ and
let ${\mathcal Q} $ be a definite rational quaternion algebra with canonical involution $\overline{\phantom{a}}$. 
Then $$\dagger : {\mathcal Q}\otimes k \to {\mathcal Q} \otimes k ;  a\otimes k \mapsto \overline{a} \otimes \overline{k} $$
defines a positive involution on $A={\mathcal Q}\otimes k$.
\end{rem}

Using this involution we computed a few examples for imaginary quadratic fields $k$:
\\
We first fix the quaternion algebra ${\mathcal Q} = 
\left(\frac{-1,-1}{\QQ } \right)$ and vary the imaginary quadratic field $k=\mathbb{Q}(\sqrt{-d})$, with $-d \equiv 1 \pmod{8} $:
\begin{center}
 \begin{tabular}{|l|l|l|l|l|}
\hline
  d & Number of   & Runtime & Runtime  & Number of  \\
 & perfect forms & Vorono\"i  & Presentation & generators \\
 \hline
$7$ &$1$ & $1.24s$ & $0.42s$ & $2$\\
\hline
$31$ & $8$ & $6.16s$ & $0.50s$& $3$ \\
\hline
$55$ & $21$ & $14.69s$ & $1.01s$ & $5$ \\
\hline
$79$ & $40$ &$ 28.74s$ & $1.78s$ &$5$ \\
\hline
$95$ & $69$ & $53.78s$ &$2.57s$ &$7$ \\
\hline 
$103$& $53$ &$38.39s$ & $2.52s$ & $6$ \\
\hline
$111$ & $83$ & $66.16s$ & $3.02s$ & $6$ \\
\hline
$255$ & $302$ & $323.93s$ & $17.54s$ & $16$ \\
\hline 
 \end{tabular}
\end{center}

In the next example we fix the imaginary quadratic field $k$ to be $\mathbb{Q}(\sqrt{-7})$
and vary the rational quaternion algebra ${\mathcal Q}$ to obtain
$A = \left(\frac{a,b}{\mathbb{Q}(\sqrt{-7})}\right)$: 

\begin{center}
 \begin{tabular}{|l|l|l|l|l|l|}
\hline
  a,b & Norm of & Number of   & Runtime & Runtime  & Number of  \\
 & discriminant & perfect forms & Vorono\"i  & Presentation & generators \\
 \hline
$-1,-1$ & $4$ & $1$ & $1.24s$ & $0.42s$ & $2$ \\
\hline
$-1,-11$ & $121$ & $20$ & $21.61s$ & $4.13s$ & $6$ \\
\hline
$-11,-14$ & $484$ & $58$ & $51.46s$ & $5.11s$ & $10$ \\
\hline
$-1,-23$ & $529$ & $184$ & $179.23s$ & $89.34s$ & $16$ \\
\hline
\end{tabular}
\end{center}

\subsection{A division algebra of index 3} \label{ex:div3}

Let $\vartheta = \zeta _9 + \zeta _9^{-1}$ be
a real root of 
$x^3-3x+1\in \Q[x]$. 
Let $ A $ be the rational division 
algebra generated (as an algebra over $\Q $)
 by $Z:=\diag(\vartheta, \sigma(\vartheta),\sigma^2(\vartheta))$ and $\Pi := \left( \begin{array}{ccc} 0 & 1 & 0 \\ 0 & 0 & 1 \\ 2 & 0 & 0 \end{array} \right)$ where $\sigma $ generates the Galois group of 
$\Q[\vartheta ]$ over $\Q $. 
As the minimal polynomial of $\vartheta $ (and hence of $Z$)
 is congruent modulo 2 to
the minimal polynomial over $\Q_2$ of a seventh root of unity, $\Pi ^3=2$
and $\Pi Z \Pi^{-1} - Z^2 = 2 $, 
we see that the Hasse invariant of $A$ is $\frac{1}{3}$ 
at the prime 2.
We use \cite{NebeSteel} to compute a maximal order $\Lambda $ 
in $A$ and its discriminant 
$2^63^6$. So the only other ramified prime in $A$ is 3 and
its Hasse invariant is $\frac{2}{3}$. 
Let $\Gamma :=\Lambda ^{\times }$, $L=\Lambda $, $A \hookrightarrow
 A_{\R }$ via one of the embeddings of $\Q[\vartheta ] \hookrightarrow \R $. 
Then $\Lambda ^{\times }$ has 431 orbits on the set of $L$-perfect
forms in ${\mathcal P}$. After reducing the presentation 
obtained by the algorithm above with standard Magma programs 
we obtain that
$\Gamma = \langle a,b\rangle $ where
\begin{equation*} a:=\frac{1}{3} \left( \begin{array}{ccc}
-\vartheta^2 - 3\vartheta + 1 &  \vartheta^2 + 2 & -\vartheta^2 + 1 \\
2\vartheta^2 + 2\vartheta - 6 &  -2\vartheta^2 + \vartheta + 3 &  -\vartheta^2 -
    \vartheta + 6 \\
2\vartheta + 8 & -2\vartheta - 2 & 3\vartheta^2 + 2\vartheta - 7 
\end{array} \right), \end{equation*}
\begin{equation*}
b:=\frac{1}{3} \left( \begin{array}{ccc}
\vartheta^2 - 2\vartheta - 3 & -2\vartheta + 1 & -\vartheta^2 + 1 \\
2\vartheta^2 + 2\vartheta - 6 & -3\vartheta^2 - \vartheta + 5 & -2\vartheta^2 + 5 \\
4\vartheta^2 + 4\vartheta - 6 & -2\vartheta - 2  & 2\vartheta^2 + 3\vartheta - 5
\end{array} \right) \end{equation*}
with defining relators
$$\begin{array}{l} 
b^2  a^2  (b^{-1}  a^{-1})^2  , \\
    b^{-2}  (a^{-1}  b^{-1})^2   a  b^{-2}  a^2  b^{-3} , \\
    a  b^2  a^{-1}  b^3  a^{-2}  b  a  b^3 , \\
    a^2  b  a  b^{-2}  a  b^{-1}  (a^{-2}  b)^2 , \\
    a^{-1}  b^2  a^{-1}  b^{-1}  a^{-5}  b^{-2}  a^{-3}, \\
    b^{-2}  a^{-2}  b^{-1}  a^{-1}  b^{-1}  a^{-2}  b^{-1}  a^{-1} 
    b^{-2}  (a^{-1}  b^{-1})^3  .
\end{array} $$

\begin{rem}
Note that  $ A^{\op } =  A^{\Tr} $  
has Hasse invariant $\frac{2}{3}$ at 2 and $\frac{1}{3}$ at 3. 
As maximal orders of $A$ and $A^{\op }$ 
correspond to each other by transposing matrices, also their
unit groups are isomorphic (via $g\mapsto g^{-\Tr }$). 
Computing the Vorono\"{\i}-tessellation for the transposed matrices,
however, we find 410 perfect forms instead of 431, which shows that 
there is no direct correspondence on the level of perfect forms. 
\end{rem}

\subsection{A matrix ring over a quaternion algebra}

Consider the rational quaternion algebra $K=\left( \frac{-1,-3}{\QQ} \right)$, ramified at $3$ and the infinite place. Let $\mathcal{O}$ be the maximal order with $\ZZ$-basis $\left \{1,~ i , ~ \frac{1}{2}(i+k), ~ \frac{1}{2}(1+j) \right \}$ and let $A=K^{2\times 2}$, $\Lambda=\mathcal{O}^{2\times 2}$. The algebra $A$ is of interest as a direct summand of the rational group algebra of $\mathrm{SL}_2(5)$.

Our algorithm finds one perfect form with automorphism group of order $720$ (isomorphic to $\mathrm{SL}_2(9)$) and a presentation of $\Lambda^\times$ on the two generators
\begin{equation*}
 \frac{1}{2}\begin{pmatrix}2&0 \\ -1+i+j-k & -i-k \end{pmatrix}, ~ \frac{1}{2} \begin{pmatrix} i+k & -1+j \\ -2i & 2-i+k \end{pmatrix},
\end{equation*}
which have orders $4$ and $6$, respectively. These two generators satisfy a set of $64$ relations, which is too large to be printed here. The commutator factor group $\Lambda^\times / (\Lambda^\times)'$ is cyclic of order $4$.

%
%

\section{Implementation}

While many things in the implementation of our algorithms are straightforward, there are some tasks which do not have an obvious solution. We present these here.

\subsection{Minimal vectors}\label{imp:minvecs}

Let $F\in \mathcal{P}$ be a form. In order to compute $S_L(F)$, the set of $L$-minimal vectors of $F$, we associate to $F$ a $\ZZ$-lattice $L_F$ and compute the minimal vectors of that lattice, e.g. using \texttt{Magma} \cite{magma}.

Let $\mathcal{B}$ be a $\ZZ$-basis of $L$. We associate to the form $F$ the following bilinear form $b_F$ on $V_\RR$:
\begin{equation*}b_F ~:~ V_\RR\times V_\RR \to \RR, ~ (x,y)\mapsto \frac{1}{2}\left(F[x+y]-F[x]-F[y]\right) \end{equation*}
Clearly, $b_F$ is positive definite since $F\in\mathcal{P}$. Now let $L_F$ be the $\ZZ$-lattice which has as its Gram matrix the Gram matrix of $b_F$ with respect to the basis $\mathcal{B}$. Then, since $F[\ell]=b_F(\ell, \ell)$ for all $\ell\in L$, the minimal vectors of $L_F$ are the coordinates of the minimal vectors of $F$ with respect to the basis $\mathcal{B}$.

\subsection{Isometry testing and automorphism groups}

We want to decide algorithmically if two forms are in the same orbit under the action
\begin{equation*}
 \Lambda^\times\times \mathcal{P}\to\mathcal{P}, ~ (\lambda , F ) \mapsto \lambda^\dagger F \lambda
\end{equation*}
of the unit group $\Lambda^\times$.

First, consider the case where $A=K$ is a division algebra. Then we choose $L=\Lambda$.
 In particular, the minimal vectors of our forms are elements of the order $\Lambda$. 

\begin{lemma}
Let $F_1,F_2\in {\mathcal P}$, $\ell  \in S_L(F_1)$.
 There is a $\lambda \in \Lambda^\times$ satisfying $\lambda^\dagger F_2 \lambda = F_1$ 
if and only if there is some $\ell_2 \in S_L(F_2)$ such that $(\ell_2 \ell^{-1})^\dagger F_2 \ell_2 \ell^{-1} = F_1$ and $\ell_2\ell^{-1}\in\Lambda^\times$.\\
 Also, we have 
 \begin{equation*}\Stab_{\Lambda^\times}(F_1)=\{\ell_1 \ell^{-1} ~|~ (\ell_1 \ell^{-1})^\dagger F_1 \ell_1 \ell^{-1} = F_1, ~\ell_1\ell^{-1}\in\Lambda^\times, ~\ell_1 \in S_L(F_1)\}.\end{equation*}
\end{lemma}
\begin{proof}
 If we have $\lambda^\dagger F_2 \lambda = F_1$, then $\lambda^{-1}S_L(F_2)=S_L(F_1)$, which is easily verified. This proves both claims.
\end{proof}

This lemma allows us to check for isometry and to compute automorphism groups of forms using only our knowledge of the finite sets of minimal vectors and without any a-priori knowledge of $\Lambda^\times$.
The membership $\lambda \in \Lambda^\times$ is tested by checking that both $\lambda $ and 
$\lambda ^{-1} $ are in the free abelian group $\Lambda $.

We now turn to the general case, where $A$ is a simple algebra  and 
$\mathfrak M = \End _{\OO }(L )$ for some $\Lambda $-lattice $L$ in $V$. 
We will only describe the computation of the automorphism group of a form $F\in\mathcal{P}$. Isometry testing is completely analogous.
We use the notation from \ref{imp:minvecs}.

Consider the $\QQ$-linear representation of $A$ induced by the action of $A$ on $V$ with respect to the basis $\mathcal{B}$. If we compute the automorphism group of the $\ZZ$-lattice $L_F$ with the Plesken-Souvignier algorithm \cite{PS97}, the result of this is a finite group of $|\mathcal{B}|\times|\mathcal{B}|$-matrices isomorphic to $\mathrm{Aut}_\ZZ(L_F)$. However, in general not all of these matrices will be contained in the image of $A\hookrightarrow \QQ^{|\mathcal{B}|\times|\mathcal{B}|}$. In order to compute only those automorphisms which satisfy this additional condition we make use of the fact that the Plesken-Souvignier algorithm may be given an additional input, namely a list of matrices which is to be fixed by the resulting lattice automorphisms. 

\begin{lemma}
 Let $\{b_1,...,b_\delta\}$ be a basis of the centralizer of the image $A\hookrightarrow \QQ^{|\mathcal{B}|\times|\mathcal{B}|}$. The matrices $X\in\GL_{|\mathcal{B}|}(\ZZ)$ stabilizing $L_F$ and fixing $\mathrm{Gram}(L_F)\cdot b_1,...,\mathrm{Gram}(L_F)\cdot b_\delta$ are contained in the image of $A$.
\end{lemma}

\begin{proof}
 Let $G:=\mathrm{Gram}(L_F)$ and consider $X\in\ \GL_{|\mathcal{B}|}(\ZZ)$ as in the statement. Then we have $X^{tr}GX=G$ and consequently, for all $1\leq i \leq \delta$, $X^{tr}Gb_i X=Gb_i=X^{tr}GX b_i$. This implies $b_iX=Xb_i$ for all $i$. It now follows from the double-centralizer theorem \cite[(7.11)]{reiner} that $X$ is contained in the image of the representation of $A$.
\end{proof}

Since $\mathfrak M=\End_\mathcal{O}(L)$, the matrices $X$ from the previous lemma are actually contained in the image of $\mathfrak M$ under the representation we considered. 
Therefore, using this lemma, we can compute the stabilizer $\Stab_{\mathfrak M^\times}(F)$ 
and its intersection 
$\Stab _{\Lambda ^*} (F) = \Stab _{\mathfrak M^{\times }} (F) \cap \Lambda $ 
without prior knowledge of $\Lambda^\times$.

\subsection{Outline of an implementation}

A detailed exposition of Vorono\"{\i}'s algorithm is beyond the scope of this paper, so we refer the reader to \cite{MR1957723,MR1881760} for details on the theory of the algorithm. The purpose of this section is to provide an overview of the steps necessary to implement our algorithms.

First of all, it should be noted that one cannot carry out precise computations in $A_\RR = A\otimes_\QQ \RR$. 
However in order to carry out the computations it is only necessary to embed $A$ into a semisimple algebra
${\mathcal A}$ with positive involution.
 If, for example, $A$ is a quaternion algebra with positive involution, we compute in $A$ itself rather than in $A_\RR$. 
In the case of the example in section \ref{ex:div3} we use an embedding of $A$ into 
${\mathcal A} = \Q[\vartheta ]^{3\times 3}$, the involution in that case being transposition. 
Note that one never needs to compute in $A_{\RR }$, as all perfect forms already 
lie in ${\mathcal A}$, however 
the number of perfect forms depends on the choice of  this algebra ${\mathcal A}$.

Assuming the reader has a suitable version of Vorono\"{\i}'s algorithm at his disposal, we now move on to the computation of a presentation of $\Lambda^\times$. As a by-product of the algorithm, we already have the facets of the Vorono\"{\i} domains. The codimension-$2$-faces are easily computed as intersections of facets. These codimension-$2$-faces (to which we will refer as "ridges" in what follows) correspond to the $2$-cells of the CW-complex described in section \ref{wr} and for computational simplicity we perform our calculations using the faces.

Notice that the side transformations (\cf Remark \ref{sidetransformation}) obtained by Vorono\"{\i}'s algorithm together with the stabilizers of the perfect forms generate $\Lambda^\times$. This is the set of generators we use in our implementation. However, this set of generators does not coincide with the generators described in Section \ref{BassSerre}.

Let $\mathcal{V}$ be a set of representatives of perfect forms obtained from Vorono\"{\i}'s algorithm. In order to compute the cycle relation corresponding to an ridge $\mathfrak{e}$ contained in the boundary of the Vorono\"{\i} domain $D_P$ of a perfect form $P=P_0 \in \mathcal{V}$ one should proceed as follows. Note that there is a finite sequence of perfect forms $P_i\in\mathcal{P}$, $0\leq i \leq \ell$, such that $D_{P_i}$ and $D_{P_{i+1}}$ meet precisely in a facet, $\mathfrak{e}\subset D_{P_i}$ for all $i$ and $P_0=P_\ell$. We now proceed iteratively: In the $i^{\text{th}}$ step we construct an element $g\in\Lambda^\times$ as a product of side-transformation (i.e. generators of $\Lambda^\times$) such that $g^\dagger P_i g \in \mathcal{V}$. So in the $\ell^{\text{th}}$ step we will have $g^\dagger P_\ell g \in \mathcal{V}$, which means $g\in\Stab_{\Lambda^\times}(P)$, yielding a relation. Start this calculation by identifying a facet $\mathfrak{f}$ of $D_P$ containing $\mathfrak{e}$, set $P_1$ to be the perfect form contiguous to $P$ through the facet $\mathfrak{f}$, $g_1$ the corresponding side transformation and $g:=g_1$. In the $i^{\text{th}}$ step, identify the facet $\mathfrak{f}$ of the 
Vorono\"{\i} domain of $P_{i-1}$ containing $\mathfrak{e}$ and satisfying $\mathfrak{f}\cap D_{P_{i-2}} \neq \mathfrak{f}$ (this can be done by computing in the 
Vorono\"{\i} domain of $g^\dagger P_{i-1} g \in \mathcal{V}$), and let $P_i$ be the perfect neighbour of $P_{i-1}$ through $\mathfrak{f}$, $g_i$ the side transformation corresponding to the facet $g^{-1}\mathfrak{f}g^{-\dagger}$ of $D_{g^\dagger P_{i-1}g}$ and replace by $g$ by $g g_i$.

The remaining relations described in Section \ref{BassSerre} are easily computed.
Finally it is useful to employ a computer algebra system capable of handling finitely presented groups in order to simplify the presentation.

The implementation of the algorithm to solve the word problem is straightforward since it merely amounts to computing the intersection of an affine line with an affine polytope.






\begin{thebibliography}{CJLdR04}

\bibitem[Ash84]{MR747876}
A. Ash, \emph{Small-dimensional classifying spaces for arithmetic subgroups
  of general linear groups}, Duke Math. J. \textbf{51} (1984) 459--468.
  
\bibitem[BDH96]{qhull}
C.B. Barber, D.P.  Dobkin, H.T. Huhdanpaa,  \emph{The Quickhull algorithm for convex hulls}, ACM Trans. on Mathematical Software, 22(4) (1996) 469--483, \url{http://www.qhull.org}
 

\bibitem[Bas93]{MR1239551}
H. Bass, \emph{Covering theory for graphs of groups}, J. Pure Appl. Algebra
  \textbf{89} (1993) 3--47.  
  
\bibitem[BCP97]{magma}
W.Bosma, J. Cannon, C. Playoust, \emph{The Magma algebra system I: The user language}, Journal of Symbolic Computation 24.3 (1997) 235--265.

\bibitem[Bro84]{MR739633}
K.~S. Brown, \emph{Presentations for groups acting on simply-connected
  complexes}, J. Pure Appl. Algebra \textbf{32} (1984) 1--10.

\bibitem[Br\"u98]{ThesisHerbert} 
H. Br\"uckner, \emph{Algorithmen f\"ur endliche aufl\"osbare Gruppen und Anwendungen},
Thesis, RWTH Aachen University, 1998. 
\bibitem[CJLdR04]{MR2073916}
C.Corrales, E. Jespers, G. Leal,  A. del R{\'{\i}}o,
  \emph{Presentations of the unit group of an order in a non-split quaternion
  algebra}, Adv. Math. \textbf{186} (2004) 498--524.
  
\bibitem[CN14]{CN}
R. Coulangeon, G. Nebe, \emph{Maximal finite subgroups and minimal classes}, to appear in Enseign. Math. 

\bibitem[Kle94]{MR1309127}
E. Kleinert, \emph{Units of classical orders: a survey}, Enseign. Math. (2)
  \textbf{40} (1994) 205--248. 

%

\bibitem[Mac64]{MR0160848}
A.~M. Macbeath, \emph{Groups of homeomorphisms of a simply connected space},
  Ann. of Math. (2) \textbf{79} (1964) 473--488.

\bibitem[Mar03]{MR1957723}
J. Martinet, \emph{Perfect lattices in {E}uclidean spaces}, Grundlehren
  der Mathematischen Wissenschaften [Fundamental Principles of Mathematical
  Sciences], vol. 327, Springer-Verlag, Berlin, 2003.

\bibitem[NeS09]{NebeSteel} 
G. Nebe, A. Steel, \emph{Recognition of Division Algebras}
 J. Algebra 322 (2009) 903--909. 
 
\bibitem[Opg01]{MR1881760}
J. Opgenorth, \emph{Dual cones and the {V}oronoi algorithm},
  Experiment. Math. \textbf{10} (2001) 599--608. 
  
\bibitem[PS97]{PS97}
W. Plesken, B. Souvignier, \emph{Computing isometries of lattices}, Journal of Symbolic Computation 24.3 (1997) 327--334.

\bibitem[Rei75]{reiner}
I. Reiner, \emph{Maximal orders}, Academic press, London, 1975.

\bibitem[Ry{\v{s}}70]{MR0276873}
S.~S. Ry{\v{s}}kov, \emph{The polyhedron {$\mu (m)$} and certain extremal
  problems of the geometry of numbers}, Dokl. Akad. Nauk SSSR \textbf{194}
  (1970) 514--517.

\bibitem[Sch09a]{MR2466406}
A. Sch{\"u}rmann, \emph{Computational geometry of positive definite
  quadratic forms}, University Lecture Series, vol.~48, American Mathematical
  Society, Providence, RI, 2009, Polyhedral reduction theories, algorithms, and
  applications.

\bibitem[Sch09b]{MR2537111}
\bysame, \emph{Enumerating perfect forms}, Quadratic forms---algebra,
  arithmetic, and geometry, Contemp. Math., vol. 493, Amer. Math. Soc.,
  Providence, RI, 2009, pp.~359--377.

\bibitem[Ser77]{MR0476875}
J.-P. Serre, \emph{Arbres, amalgames, {${\rm SL}_{2}$}}, Soci\'et\'e
  Math\'ematique de France, Paris, 1977,
Ast{\'e}risque, No. 46.
  
\bibitem[Shi65]{Shimizu}
H. Shimizu, \emph{On zeta functions of quaternion algebras.}
 Ann. of Math.  81 (1965) 166--193.
  
\bibitem[Sou78]{MR0470141}
C. Soul{\'e}, \emph{The cohomology of {${\rm SL}_{3}({\bf Z})$}},
  Topology \textbf{17} (1978) 1--22.

\bibitem[Ste07]{MR2289048}
W. Stein, \emph{Modular forms, a computational approach}, Graduate Studies
  in Mathematics, vol.~79, American Mathematical Society, Providence, RI, 2007,
  With an appendix by Paul E. Gunnells. 

\bibitem[Vor07]{Vo1}
G.~F. Vorono\"{\i}, \emph{Nouvelles applications des param\`etres continus \`a
  la th\'eorie des formes quadratiques : 1 sur quelques propri\'et\'es des
  formes quadratiques parfaites}, J. Reine Angew. Math. \textbf{133} (1907)
  97--178.

\bibitem[WYH13]{MR3074816}
Takao Watanabe, Syouji Yano, and Takuma Hayashi, \emph{Vorono\"\i 's reduction
  theory of {$GL_n$} over a totally real number field}, Diophantine methods,
  lattices, and arithmetic theory of quadratic forms, Contemp. Math., vol. 587,
  Amer. Math. Soc., Providence, RI, 2013, 213--232. 
  
\bibitem[Yas10]{MR2721434}
D. Yasaki, \emph{Hyperbolic tessellations associated to {B}ianchi groups},
  Algorithmic number theory, Lecture Notes in Comput. Sci., vol. 6197,
  Springer, Berlin, 2010, 385--396.

\end{thebibliography}

\providecommand{\bysame}{\leavevmode\hbox to3em{\hrulefill}\thinspace}
\providecommand{\MR}{\relax\ifhmode\unskip\space\fi MR }
 
\providecommand{\MRhref}[2]{%
  \href{http://www.ams.org/mathscinet-getitem?mr=#1}{#2}
}
\providecommand{\href}[2]{#2}

\end{document}